\documentclass[12pt]{iopart}
\usepackage{amsthm,amssymb}
\usepackage{graphicx,color}
\usepackage{subfigure}
\usepackage{multirow}
\usepackage{tikz}
\usepackage{subfigure}
\newtheorem{theorem}{Theorem}[section]

\newtheorem{remark}[theorem]{Remark}

\def\p{\partial}
\def\na{\nabla}

\def\x{{\mathbf x}}

\def\f{\frac}
\def\q{\quad}
\def\Om{\Omega}
\def\om{\omega}
\def\ep{\epsilon}

\def\mS{{\mathcal S}}
\def\mD{{\mathcal D}}
\def\mK{{\mathcal K}}
\def\mT{{\mathcal T}}
\def\mC{{\mathcal C}}
\def\mE{{\mathcal E}}
\def\mL{{\mathcal L}}
\def\Q{\mathbf Q}
\def\P{\mathbf P}
\def\a{\mathbf a}
\def\z{\mathbf z}
\def\eref#1{(\ref{#1})}
\newcommand{\bm}[1]{\mbox{\boldmath{$#1$}}}
\begin{document}
\title[]{Mathematical framework for multi-frequency identification of thin insulating and small conductive inhomogeneities}
\author{Habib Ammari$^1$, Jin Keun Seo$^2$ and Tingting Zhang$^2$}
\address{$^1$ Department of Mathematics, ETH Z\"urich, R\"amistrasse 101, 8092 Z\"urich, Switzerland.}
\address{$^2$ Department of Computational Science and Engineering, Yonsei University, Seoul 120-749
Korea}
\ead{habib.ammari@math.ethz.ch,\{seoj,zttysu\}@yonsei.ac.kr}
\vspace{10pt}
\begin{abstract}
We are aiming to identify the thin insulating inhomogeneities and small conductive inhomogeneities inside an electrically conducting medium by using multi-frequency electrical impedance tomography (mfEIT). The thin insulating inhomogeneities are considered in the form of tubular neighborhood of a curve and small conductive inhomogeneities are regarded as circular disks. Taking advantage of the frequency dependent behavior of insulating objects, we give a rigorous derivation of the potential along thin insulating objects at various frequencies. Asymptotic formula is given to analyze relationship between inhomogeneities and boundary potential at different frequencies. In numerical simulations, spectroscopic images are provided to visualize the reconstructed admittivity at various frequencies. For the view of both kinds of inhomogeneities, an integrated reconstructed image based on principle component analysis (PCA) is provided. Phantom experiments are performed by using Swisstom EIT-Pioneer Set.
\end{abstract}
\maketitle
\section{Introduction}
Multi-frequency electrical impedance tomography (mfEIT) is a noninvasive method to provide the images of  the conductivity  and permittivity  at various frequencies ranging from 1 kHz to 1 MH inside an electrically conducting object \cite{Griffiths1992,Halter2008,Jain1997,Seo2008a,Wilson2001}.  Recently, multi-frequency imaging techniques have been paid considerable attention due to their advantages in probing thin insulating inhomogeneities with their widths. Their potential applications include biomedical imaging to probe  biological tissues comprising  insulating cell membranes \cite{Fricke1924,Gabriel1996b,Holder2005,Pavlin2003,Kim2012,Seo2012} and non-destructive testing (NDT) to probe insulating  cracks \cite{McCann2001,Pour2011b,Pour2012,Kim2009,Zhang2015}.

In mfEIT, we inject ac currents using surface electrodes to produce the time-harmonic electrical field inside the imaging object.  The induced electrical field is determined by  the distribution of effective conductivity and permittivity, the geometry of the object, electrode positions, the applied frequency, and others \cite{Seo2012}. In the presence of thin insulating inhomogeneities or insulating membranes, the time-harmonic electrical potential near these inhomogeneities varies significantly with the applied frequency.  In particular, the potential jump across the thin insulating inhomogeneities changes a lot with frequency, because currents pass through  a thin insulating inhomogeneity as frequency increases. Those variations with frequency are conveyed to the boundary voltages \cite{Kim2009,Kim2012,Zhang2015}.  Multi-frequency EIT (mfEIT) can take advantage of this frequency-dependent behavior to measure  tissue anomalies and thin insulating inhomogeneities with their thickness.

One challenging issue is how to link the spectroscopic information obtained from  mfEIT to the structural information of an imaging object containing thin insulating inhomogeneities.  We need to describe the role of the thin insulating inhomogeneities by understanding the frequency-dependent interplays between the real and imaginary parts of the complex potentials due to the change of the  refraction along the thin insulating inhomogeneities as frequency varies.

This paper deals with this challenging issue of the spectroscopic mfEIT by means of rigorous mathematical analysis with both numerical and experimental validations. We start with the simplest model of a linear-shaped thin insulating inhomogeneity with a uniform thickness, in order to provide a basis of the rigorous expression on the frequency-dependent behavior of the potential due to the influence of thin insulating inhomogeneities. Through a rigorous asymptotic analysis, we provide the behavior of the time-harmonic potential $u^\om$ across the linear inhomogeneity  with respect to  angular frequency $\om$. To be precise, the frequency-dependent variation of the exterior normal derivative $\f{\p u^\om}{\p\nu}$ on the boundary of the inhomogeneity can be approximately described in terms of the thickness ($2\delta$), the  contrast ratio between the background admittivity $\sigma_b+i\om\ep_b$ and the inhomogeneity permittivity $\ep_c$:
$$
\left|\f{\p u^\om}{\p\nu}\right|~~ \approx ~~ \f{\om}{2\delta }  ~ \f{\left|[u^\omega]\right| }{ |(\sigma_b+i\omega\epsilon_b)/\ep_c|}.
$$

Assuming the above approximation for a more general shape of thin insulators, we could get better understanding on spectroscopic admittivity images being reconstructed from standard mfEIT algorithm. In a special case of inclusions (highly conducting disks and linear insulating segments) inside the imaging object, we derive a formula for detecting both inhomogeneities. Finally, various numerical simulations are performed at multiple frequencies to give spectroscopic reconstructed images. Using spectroscopic images, an integrated image based on principle component analysis (PCA) is presented. In addition, we take use of Swisstom EIT-Pioneer set to carry out phantom experiments at a frequency range from 50kHz to 250kHz.

\section{Mathematical model}
For rigorous analysis, we use the simplified two-dimensional
model by considering axially symmetric cylindrical sections under
the assumption that the out-of-plane current density is negligible
in an imaging slice. We assume a two dimensional electrically
conducting domain $\Om$ with its connected $C^2$-boundary $\partial \Om$.
We denote the conductivity distribution of the domain by $\sigma$
and the permittivity distribution by $\epsilon$. Inside $\Om$,
there exist thin insulating objects $\mC_k,~k=1,2,\ldots, N_C$ and small conductive objects $D_k,~ k=1,2,\ldots, N_D$. Let $D=\cup_{k=1}^{N_D} D_k$ and
$\mC=\cup_{k=1}^{N_C} \mC_k$ denote the collections of the
conductive objects and thin insulating objects, respectively. Since the
conductivity $\sigma$ and permittivity $\ep$ change abruptly
across the thin objects and conductive objects, we denote
\begin{equation}\label{Eq:Admt}
\gamma^\omega(x)=
\left\{\begin{array}{ll}
\gamma^\om_c = \sigma_c+i\omega\epsilon_c &\q \mbox{for}~x\in  \mathcal{C},\\
 \gamma^\om_d  =\sigma_d+i\omega\epsilon_d &\q \mbox{for}~x\in D, \\
  \gamma^\om_b =\sigma_b+i\omega\epsilon_b &\q\mbox{for}~x\in\Omega\backslash(D \cup \mathcal{C}), \\
\end{array}
\right.
\end{equation}
Because the thin objects $\mC_k$ are highly insulating, we consider the following extreme contrast case:
$$
\sigma_c/\sigma_b\approx 0.
$$
In the frequency range below 1MHz
($\f{\om}{2\pi}\le 10^6$), we inject a sinusoidal current
$g(x)\sin (\om t)$ at $x\in\p\Om$ where $g$ is the magnitude of
the current density on $\p\Om$ and $g\in
H^{-1/2}_\diamond(\p\Om)$. Here $H^{-1/2}_\diamond(\p\Om):=\{ \phi\in H^{-1/2}(\p\Om)~:~
\langle \phi, 1 \rangle =0\}$ with $\langle \, , \, \rangle$ being the duality pair between
$H^{-1/2}(\p\Om)$ and $H^{1/2}(\p\Om)$. The injected current produces the
time-harmonic potential $u^\om$ in $\Om$ which is dictated by
\begin{equation}\label{Eq:uw}
\left\{
\begin{array}{ll}
 \nabla\cdot\left(\gamma^\omega\nabla u^\omega\right) =0 &\quad\mbox{in}~\Omega,\\
 \gamma^\omega \f{\partial u^ \omega}{\p \nu} = g &\quad\mbox{on}~\partial\Omega,
 \end{array}\right.
\end{equation}
where $\gamma^{\om} = \sigma + i\omega \ep$ is the admittivity distribution, ${\nu}$ is the
outward unit normal vector on $\p\Om$, and $\f{\p}{\p\nu}$ is the
normal derivative. Setting $\int_{\p\Om} u^\om ds=0$, we can
obtain a unique solution $u^\om$ to (\ref{Eq:uw}) from the
Lax-Milgram theorem. Hence, we can define the Neumann-to-Dirichlet
map $\Lambda_\omega:H^{-1/2}_\diamond(\p\Omega) \to
H^{1/2}_\diamond(\p\Omega)$ by
$\Lambda_\omega(g)=u^\om|_{\p\Omega}$. Using $N_E-$channel
multi-frequency EIT system, we may inject $N_E$ number of linearly
independent currents at several angular frequencies
$\om_1,\ldots,\om_{N_\om}$ and measure the induced corresponding
boundary voltages. We collect these current-voltage data
$\{\Lambda_{\om_j}(g_k)~:~~k=1,\ldots, N_E,~j=1,\ldots,
N_\omega\}$ at various frequencies ranging from 10Hz to 1MHz. The inverse problem is to identify the thin insulators $\mC_k$ and
small conductive objects $D_k$ from measured current-voltage data
in multi-frequency EIT system.

To carry out detailed analysis, we will restrict our considerations to geometric structures of $\mC$ and $D$. We assume that each thin insulating inhomogeneity $\mC_k$ has a uniform thickness of $\delta_k$ and is a neighborhood of a $C^3-$smooth open curve $\mathcal L_k$:
\begin{equation}\label{Eq:curve2}
  \mathcal{C}_k = \{x+h \nu_x~:~ x\in \mathcal L_k, ~-\delta_k<h<\delta_k  \},~~~(k=1,2,\ldots,N_C),
\end{equation}
where $\nu_x$ is the unit normal vector at $x$ on the curve $\mathcal{L}_k$. The thickness to the length ratio is assumed to be very small, that is,  $\delta_k\approx 0$. We also assume that each small conductive object has the form
\begin{equation}\label{Dk}
D_k :=z_k + \delta_D B_k,~~~ (k=1,2,\ldots,N_D),
\end{equation}
where $z_k$ is the center of $D_k$, $B_k$ is a bounded smooth reference domain centered at $(0,0)$ and $\delta_D$ is related to the diameter of $D_k$. We assume that $\mC_k$ and $D_k$ are well separated from each other as well as from the boundary $\p\Omega$, {\it i.e.}, the separation distance is much larger than the characteristic size of the conductors or thickness of the insulators. Note that in the non-resolved case, where the distance between small conductors  or insulators is of order  of the characteristic size of $D_k$ or thickness of $\mC_k$, the conductors and the insulators can not be determined separately. Only equivalent targets can be imaged using boundary measurements (see \cite{ammari_num}). Therefore, throughout this paper, we assume that there exists a positive constant $d_0>0$ \cite{Ammari2006,Ammari2004} such that
 \begin{equation}\label{separation}
    \begin{array}{c}
  \inf_{k\neq k^\prime}\mbox{dist}(D_k,D_{k^\prime})\geq d_0, ~~\inf_{k\neq k^\prime}\mbox{dist}(\mathcal{C}_k,\mathcal{C}_{k^\prime})\geq d_0, \\
  \mbox{dist}(\mC,\p\Omega)\geq d_0, ~~\mbox{dist}(D,\p\Omega)\geq d_0,~~\mbox{dist}_{k,j}(\mC_k,D_j)\geq d_0.
    \end{array}
\end{equation}
\section{Methods}
 In this part, we will focus on rigorous analysis of the frequency-dependent behaviors of the complex potential $u^\omega$ around the thin insulating inhomogeneities. We will derive an explicit formula for detecting positions of conductive inhomogeneities and thin insulating inhomogeneities by using asymptotic expansions of $u^\omega$. The explicit formula depends on the operating frequency $\omega$ and the insulator thickness $\delta_k$.

 To start with, since each thin object $\mathcal C_k$ is highly non-conductive, there is a noticeable potential jump along the thin insulating objects \cite{Ammari2006,Kim2012,Zhang2015}. For ease of notation, we define exterior($+$)/interior($-$) normal derivative on the boundary of $\mathcal C_k$ as follows:
\[
\begin{array}{lll}
\f{\p u^\omega}{\p\nu}(x-\delta_k\nu_x )|_\pm &=& \lim_{s\rightarrow0^+}\f{\p u^\omega}{\p\nu}(x-\delta_k\nu_x \mp
s\nu_x)\\
\f{\p u^\omega}{\p\nu}(x+\delta_k\nu_x)|_\pm&=
&\lim_{s\rightarrow0^+}\f{\p u^\omega}{\p\nu}(x+\delta_k\nu_x \pm s\nu_x )
\end{array} \quad~ \q(x\in \mathcal{L}_k).
\]
Denote by $[u^\omega]_k$ and $\left[\f{\p u^\omega}{\p\nu}\right]_k$ the jump of the potential and the jump of its normal derivative across the boundary of the thin insulating object $\mC_k$, respectively:
\begin{equation}\label{jump-across}
\begin{array}{lll}
[u^\omega]_k(x)&:=&u^\omega (x+ \delta_k\nu_x ) -u^\omega (x-\delta_k\nu_x ),\\
 \left[ \f{\p u^\omega}{\p\nu}\right]_k(x)&:=& \f{\p u^\omega}{\p\nu} (x+\delta_k\nu_x )|_+-
 \f{\p u^\omega}{\p\nu} (x-\delta_k\nu_x )|_+,
 \end{array}
 \quad~ \q(x\in \mathcal{L}_k).
\end{equation}
Then, based on the above notations, we have the following theorem for jump conditions across thin insulating inhomogeneities:
\begin{theorem}\label{Th:JumpCondition}
Let $\mC_k = \{x: |x_1|\leq\ell_k, |x_2|\leq \delta_k\}$ be a tubular neighbourhood of segment $\mL_k=\{x:|x_1|\leq\ell_k, x_2=0\}$ as shown in Figure \ref{Fig:lineCrack}. Let $\mathcal{L}_k^\diamond = \{x\in\mathcal{L}_k: ~dist(x,\p\mathcal{L}_k)>c_0\}$ for $c_0>0$ and $\mC_k^\diamond = \{x+h \nu_x:~x\in \mathcal L_k^\diamond, ~-\delta_k<h<\delta_k \},~(k=1,2,\cdots,N_C)$. For $\delta_k\approx 0$, the jump of potential and the jump of normal derivative across $\mC_k^\diamond$ can be approximated by
\begin{equation}\label{Eq:NuJ}
 \begin{array}{lll}
  [u^\omega]_k(x) &=&  \f{2\delta_k}{\lambda_c(\omega)} \f{\p u^\omega}{\p\nu}(x-\delta_k\nu_x )|_+ + O(\delta_k^2\ln\delta_k)
\\
\left[ \f{\p u^\omega}{\p\nu}\right]_k(x)&=& O(\delta_k\ln\delta_k)
\end{array}
 \quad~ \q(x\in \mathcal{L}_k^\diamond),
\end{equation}
with \begin{equation}\label{Eq:lamC}
\lambda_c(\omega) =  \f{\sigma_c+i\omega\epsilon_c}{\sigma_b+i\omega\epsilon_b}.
\end{equation}
\end{theorem}
\begin{figure}[ht!]
\centering
 \includegraphics[width=10cm, height=3cm]{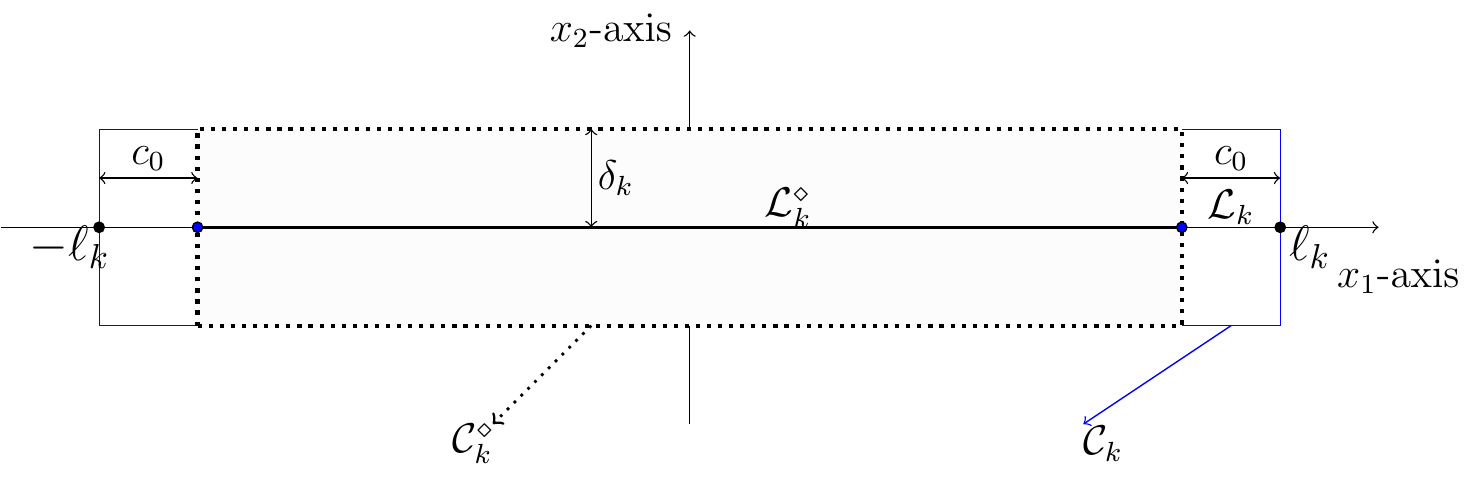}\\
  \caption{Linear object model.}\label{Fig:lineCrack}
\end{figure}
\begin{proof}
The potential $u^\om$ in \eref{Eq:uw} can be expressed as
   \begin{eqnarray*}
u^\omega(x)&=&H_k(x)+\int_{\p \mC_k}\Gamma(x,x') \phi_k (x') ds_{x'}\q\mbox{in }~\Om,
\end{eqnarray*}
where $\Gamma(x,x'):= -\f{1}{2\pi} \ln|x-x'|$ is the fundamental solution of Laplacian in two dimensions, $H_k$ is a harmonic function in a neighborhood of $\mC_k$, and
\begin{equation}\label{phik}
\phi_k = \left(\lambda_c -1\right) \f{\p u^\omega}{\p\nu}\bigg|_-= \left(1-\f{1}{\lambda_c}\right) \f{\p u^\omega}{\p\nu}\bigg|_+  \quad \mbox{on } ~\p \mC_k.
\end{equation}
To be precise, $H_k$ can be expressed in the following form:
$$H_k=-\f{1}{\gamma_b^\omega}\mS_{\Om}g  +
\mD_{\Om}(\Lambda_\om(g))
+ \sum_{j=1}^{N_D}\mS_{D_j}\psi_{D_j}
+ \sum_{j\neq k}\mS_{\mC_j}\phi_j,$$
where $\psi_{D_j} = \left(\f{\gamma_d^\omega}{\gamma_b^\omega}-1\right)\f{\p u^\omega}{\p \nu}\bigg|_-$ on ${\p D_j}$, $\mS_{\Omega}$, $\mS_{D_j}$, $\mS_{\mC_j}$  are the single layer potentials over the domain $\Omega,~D_j,~\mC_j$ respectively and $\mD_{\Omega}$ is the double layer potential over the domain $\Omega$; see \cite{Ammari2004,Ammari2007,Seo2012}.

From the transmission conditions across $\mC_k$, $\phi_k$ satisfies
\begin{equation}\label{Eq:lamC2}
\left(\lambda I-\mK_{\mC_k}^*\right)\phi_k =\f{\p H_k}{\p\nu} \quad \mbox{on } ~\p \mC_k,
\end{equation}
where $\lambda =\f{\lambda_c+1}{2(\lambda_c-1)}$ and  $\mK_{\mC_k}^* \phi $ is given by
\begin{equation}\label{Eq:komega}
  \mathcal{K}^*_{\mC_k} \phi(x) := \f{1}{2\pi}\int_{\p{\mC_k}} \f{\langle x'-x, \nu_{x} \rangle}{|x-x'|^2}\phi(x')ds_{x'},
  \quad x\in \p {\mC_k}.
\end{equation}

Now, we compute $\left[ \f{\p u^\omega}{\p\nu}\right]_k(x)$ for $x=(s,0)\in \mL_k^\diamond$.
From \eref{Eq:lamC2}, we have
\begin{equation}\label{Eq:lamC2-1}
\begin{array}{lll}
&&\left(\lambda I-\mK_{\mC_k}^*\right)\phi_k(s,\delta_k)- \left(\lambda I-\mK_{\mC_k}^*\right)\phi_k(s,-\delta_k)\\
&&~~~~~~~~~~~~~=\int_{-\delta_k}^{\delta_k}\f{d^2 H_k(s,t)}{d t^2}dt  \quad \mbox{for } ~(s,0)\in \mL_k^\diamond.
\end{array}
\end{equation}
For convenience, let $\p{\mC_k} = \p{\mC_k}^\sharp\cup \mL_k^+ \cup \mL_k^- $ where $\p{\mC_k}^\sharp=\{ x\in \p{\mC_k}: |x_1|> \ell_k-c_0/2\}$ and $\mL_k^\pm=\{x\in \p{\mC_k}: |x_1|\leq \ell_k-c_0/2, x_2=\pm\delta_k\}$ as described in Figure \ref{Fig:CrackDef}.
\begin{figure}[ht!]
\centering
 \includegraphics[width=10cm, height=3cm]{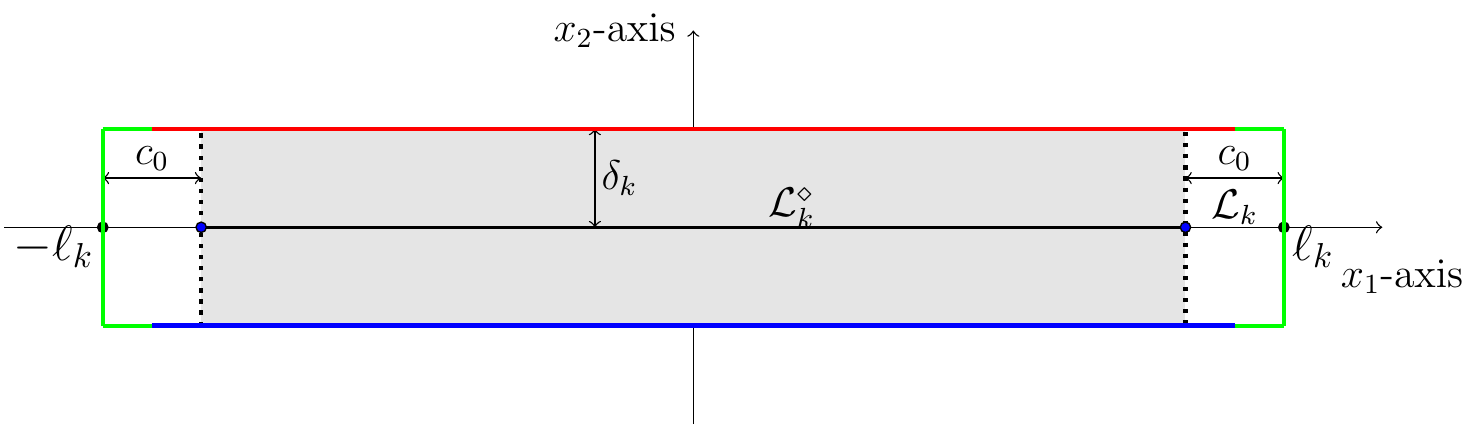}\\
  \caption{$\p{\mC_k}$ can be divided as three parts; green color stands for $\p{\mC_k}^\sharp$, red color is $\mL_k^+$ and blue color is $\mL_k^-$.}\label{Fig:CrackDef}
\end{figure}

Then, when $x=(s,-\delta_k)$, $ \mK_{\mC_k}^*\phi_k(x)$ can be written as
\begin{eqnarray*}
\hspace{-2.2cm}  2\pi\mK_{\mC_k}^*\phi_k(s,-\delta_k) &=& \int_{\p{\mC_k}^\sharp}\f{\langle x'-(s,-\delta_k), \nu_{x} \rangle}{|(s,-\delta_k)-x'|^2}\phi_k(x') dx'  \\
  & &+ \int_{\mL_k^+}\f{\langle x'-(s,-\delta_k), \nu_{x} \rangle}{|(s,-\delta_k)-x'|^2} \phi_k(x')dx'  +\int_{\mL_k^-}\f{\langle x'-(s,-\delta_k), \nu_{x} \rangle}{|(s,-\delta_k)-x'|^2} \phi_k(x')dx'.
\end{eqnarray*}
Since $\mL_k$ is a line, $x'-(s,-\delta_k)$ is perpendicular to $\nu_{x}$ for $x'\in \mL_k^-$ and $\langle x'-(s,-\delta_k),\nu_x\rangle = 2\delta_k$ for $x'\in \mL_k^+$, which directly yield
\begin{eqnarray}
\hspace{-2.2cm} 2\pi \mK_{\mC_k}^*\phi_k(s,-\delta_k) &=& \int_{\p{\mC_k}^\sharp}\f{\langle x'-(s,-\delta_k), \nu_{x} \rangle}{|(s,-\delta_k)-x'|^2}\phi_k(x') dx' + \int_{\mL_k^+}\f{\langle x'-(s,-\delta_k), \nu_{x} \rangle}{|(s,-\delta_k)-x'|^2} \phi_k(x')dx'\nonumber\\
&=&\int_{-\ell_k+c_0/2}^{\ell_k-c_0/2} \f{2\delta_k}{|s-s'|^2+(2\delta_k)^2}\phi_k(s')d s' +\int_{\p{\mC_k}^\sharp}\f{\langle x'-(s,-\delta_k), \nu_{x} \rangle}{|(s,-\delta_k)-x'|^2}\phi_k(x') dx'.\nonumber\\ \label{Eq:Kminus}
\end{eqnarray}
Analogously, it follows that
\begin{eqnarray}
\hspace{-2.2cm} 2\pi \mK_{\mC_k}^*\phi_k(s,\delta_k)
&=&\int_{-\ell_k+c_0/2}^{\ell_k-c_0/2} \f{2\delta_k}{|s-s'|^2+(2\delta_k)^2}\phi_k(s')d s' +\int_{\p{\mC_k}^\sharp}\f{\langle x'-(s,\delta_k), \nu_{x} \rangle}{|(s,\delta_k)-x'|^2}\phi_k(x') dx'.\nonumber\\ \label{Eq:Kplus}
\end{eqnarray}
since $x'-(s,\delta_k)$ is perpendicular to $\nu_{x}$ for $x'\in \mL_k^+$ and $\langle x'-(s,\delta_k),\nu_x\rangle = 2\delta_k$ for $x'\in \mL_k^-$.
Combining formulas (\ref{Eq:Kminus}) and (\ref{Eq:Kplus}) yields
\begin{equation}\label{Eq:lamC2-2}
\mK_{\mC_k}^*\phi_k(s,\delta_k)- \mK_{\mC_k}^*\phi_k(s,-\delta_k) =\mathcal{T}_{k}\phi^\sharp_k(s) +E(s)\quad \mbox{for } ~(s,0)\in \mL_k^\diamond,
\end{equation}
where
\begin{eqnarray}
&&\mathcal{T}_{k}\phi^\sharp_k(s) =\f{1}{2\pi}\int_{-\ell_k+c_0/2}^{\ell_k-c_0/2} \f{2\delta_k}{|s-s'|^2+(2\delta_k)^2}\phi_k^\sharp(s')d s',\label{Eq:TphiSharp}\\
& &\phi_k^\sharp(s) =\phi_k(s,\delta_k)-\phi_k(s,-\delta_k)=\left(1-\f{1}{\lambda_c}\right)\left[\f{\p u^\omega}{\p\nu}\right]_k(s,0), \label{Eq:phiSharp}
\end{eqnarray}
and
$E(s)$ for $(s,0)\in \mL_k^\diamond$ is given by
$$
E(s)=\f{1}{2\pi}\int_{\p{\mC_k}^\sharp }\left\{ \f{\langle x'-(s,\delta_k), \nu_{x} \rangle}{|(s,\delta_k)-x'|^2} -\f{\langle x'-(s,-\delta_k), \nu_{x} \rangle}{|(s,-\delta_k)-x'|^2} \right\}\phi_k(x')ds_{x'}.
$$
Since $H(x)$ is harmonic in a neighborhood of $\mC_k$,
\begin{equation}\label{Eq:Hs}
\sup_{(s,0)\in \mL_k^\diamond}\left|\int_{-\delta_k}^{\delta_k}\f{d^2 H_k(s,t)}{d t^2}dt\right|\le 2\delta_k \|\na\na H_k\
\|_{C^2(\mC_k)}=O(\delta_k).
\end{equation}
Moreover, the term $E(s)$ can be estimated by the mean-value theorem. Using the fact that $c_0/2\leq\mbox{dist}(\mL_k^\diamond, \p{\mC_k}^\sharp)$ and $|\langle x'-(s,\pm\delta_k), \nu_{x} \rangle|\leq 2\delta_k$, we get
\begin{equation}\label{Eq:Es}
\sup_{(s,0)\in \mL_k^\diamond}\left| E(s)\right|
\le  \f{16}{c_0^2} \|\phi_k\|_{L^1(\mC^\sharp_k)}\delta_k=O(\delta_k).
\end{equation}
Therefore, it follows from (\ref{Eq:lamC2-1})-(\ref{Eq:Es}) that
\begin{equation}\label{Eq:T}
  (\lambda I - \mathcal{T}_k) \phi_k^\sharp(s) = \int_{-\delta_k}^{\delta_k}\f{d^2 H_k(s,t)}{d t^2}dt+E(s)\quad \mbox{for } ~(s,0)\in \mL_k^\diamond.
\end{equation}
Note that
\begin{eqnarray}\label{Eq:TphiK}
   \mathcal{T}_k\phi_k^\sharp(s)&=& \f{1}{2\pi}\int_{-\ell_k+c_0/2}^{\ell_k-c_0/2} \f{2\delta_k}{|s-s'|^2+(2\delta_k)^2}(\phi_k^\sharp(s')-\phi_k^\sharp(s))d s', \nonumber\\ && + \phi_k^\sharp(s)\f{1}{2\pi}\int_{-\ell_k+c_0/2}^{\ell_k-c_0/2} \f{2\delta_k}{|s-s'|^2+(2\delta_k)^2} d s'.
\end{eqnarray}
Direct computation yields
\begin{eqnarray*}
 && \f{1}{2\pi}\int_{-\ell_k+c_0/2}^{\ell_k-c_0/2} \f{2\delta_k}{|s-s'|^2+(2\delta_k)^2} d s',\\
 &&=\f{1}{2\pi}\left(\arctan\f{\ell_k-c_0/2-s}{2\delta_k}-\arctan\f{-\ell_k+c_0/2-s}{2\delta_k}\right),\nonumber\\
 &&= \f{1}{2}+\left(-\f{2}{\ell_k-c_0/2-s}+\f{2}{-\ell_k+c_0/2-s}\right)\delta_k+O(\delta_k^3)\nonumber,\\
 && =\f{1}{2} -\f{4(\ell_k-c_0/2)}{(\ell_k-c_0/2)^2-s^2}\delta_k+O(\delta_k^3).
\end{eqnarray*}
For $s\in \mL^\diamond_k$, we have
\begin{eqnarray*}
(\ell_k-c_0/2)^2-s^2\geq   (\ell_k-c_0/2)^2-(\ell_k-c_0)^2 =c_0(\ell_k-3c_0/4).
\end{eqnarray*}
Therefore, we obtain that
\begin{eqnarray}\label{Eq:Tkpt1}
\f{1}{2\pi}\int_{-\ell_k+c_0/2}^{\ell_k-c_0/2} \f{2\delta_k}{|s-s'|^2+(2\delta_k)^2} d s' = \f{1}{2}+O(\delta_k).
 \end{eqnarray}

Since $\f{\p \gamma^\om}{\p x_1}=0$ in a neighborhood of $\mC^\diamond_k$,  $\f{\p u^\om}{\p x_1}$ is a weak solution of $\na\cdot (\gamma^\om \f{\p u^\om}{\p x_1})=0$ near the region $\mC^\diamond_k$. Hence,  $\phi_k^\sharp(s)$ is differentiable for $(s,0)\in \mL_k^\diamond$, and  $\phi_k^\sharp(s)-\phi_k^\sharp(s')$ can be estimated by
\begin{equation}\label{Eq:estPhi}
  \hspace{-1cm} |\phi_k^\sharp(s)-\phi_k^\sharp(s')|\leq \left\|\f{d}{ds}\phi_k^\sharp(s)\right\|_{L^\infty(\mL_k^\diamond)}|s-s'|\quad \mbox{for } ~(s',0),~(s,0)\in \mL_k^\diamond.
\end{equation}
From the above estimate (\ref{Eq:estPhi}), we have
\begin{eqnarray}\label{Eq:Tkpt2}
\hspace{-2.5cm}&& \left|\f{1}{2\pi}\int_{-\ell_k+c_0/2}^{\ell_k-c_0/2} \f{2\delta_k}{|s-s'|^2+(2\delta_k)^2}(\phi_k^\sharp(s')-\phi_k^\sharp(s))d s' \right|
\\
\hspace{-2.5cm}&&\quad\leq \f{\left\|\f{d}{ds}\phi_k^\sharp(s)\right\|_{L^\infty(\mL_k^\diamond)}}{2\pi}\int_{-\ell_k+c_0/2}^{\ell_k-c_0/2} \f{2\delta_k}{|s-s'|^2+(2\delta_k)^2}|s-s'|ds'=-\left\|\f{d}{ds}\phi_k^\sharp(s)\right\|_{L^\infty(\mL_k^\diamond)}\f{2\delta_k}{\pi} \ln2\delta_k\nonumber\\
\hspace{-2.5cm}&&\quad \quad +\f{\delta_k \left\|\f{d}{ds}\phi_k^\sharp(s)\right\|_{L^\infty(\mL_k^\diamond)}}{2\pi} \ln\left[((\ell_k-\f{c_0}{2}-s)^2+4\delta_k^2)((\ell_k-\f{c_0}{2}+s)^2+4\delta_k^2)\right]  \nonumber\\
\hspace{-2.5cm}&&\quad =O(\delta_k\ln\delta_k).\nonumber
\end{eqnarray}
From (\ref{Eq:Tkpt1}) and (\ref{Eq:Tkpt2}), $\mathcal{T}_k\phi_k^\sharp(s)$ in (\ref{Eq:TphiK}) can be estimated by
\begin{equation}\label{Eq:Tkpt3}
  \mathcal{T}_k\phi_k^\sharp(s) = \f{1}{2}\phi_k^\sharp(s)+O(\delta_k\ln\delta_k).
\end{equation}
Combining (\ref{Eq:phiSharp}), (\ref{Eq:T}) and (\ref{Eq:Tkpt3}), we arrive at
\begin{eqnarray}
  \left(\lambda-\f{1}{2}\right)\phi_k^\sharp(s) = O(\delta_k\ln\delta_k),\nonumber\\
  \left[\f{\p u^\omega}{\p\nu}\right]_k(s,0) = \f{\lambda_c}{\lambda_c-1}\phi_k^\sharp(s) = O(\delta_k\ln\delta_k)\quad \mbox{for } ~(s,0)\in \mL_k^\diamond.\label{Eq:LamT}
\end{eqnarray}
It easily follows from the above analysis and transmission condition that
\begin{eqnarray*}
  [u^\omega]_k(s) &=& u^\omega(s,\delta_k)-u^\omega(s,-\delta_k) = \int_{-\delta_k}^{\delta_k} \f{d }{dt}u^\omega(s,t)dt\\
  &=& 2\delta_k \f{\p u^\omega}{\p x_2}(s,-\delta_k)|_- + \int_{-\delta_k}^{\delta_k} \f{d }{dt}u^\omega(s,t)- \f{\p u^\omega}{\p x_2}(s,-\delta_k)|_{-} dt
\end{eqnarray*}
where
\begin{eqnarray*}
  &&|\int_{-\delta_k}^{\delta_k} \f{d }{dt}u^\omega(s,t)- \f{\p u^\omega}{\p x_2}(s,-\delta_k)|_{-} dt|\\
  &\leq& \int_{-\delta_k}^{\delta_k} \left|\f{d }{dt}u^\omega(s,t)- \f{\p u^\omega}{\p x_2}(s,-\delta_k)|_{-}\right| dt\leq C\delta_k^2\ln\delta_k.
\end{eqnarray*}
Then we have
\begin{equation}\label{Eq:JumpEstimate}
  [u^\omega]_k(s) = \f{2\delta_k}{\lambda_c(\omega)} \f{\p u^\omega}{\p \nu}(s,-\delta_k)|_+ + O(\delta_k^2\ln\delta_k)\quad \mbox{for } ~(s,0)\in \mL_k^\diamond.
\end{equation}
\end{proof}
\begin{remark}\label{Rk:JumpCondition}
 The approximation formulas given in \eref{Eq:NuJ} can be proved for special cases of thin insulators by using layer potential techniques. Unfortunately, the proof of (\ref{Eq:NuJ}) in general still remains a challenging issue due to technical difficulties in obtaining a uniform estimate for the Hessian matrix of $u^\omega$ in $\mathcal{C}_k^\diamond$  with respect to $\delta_k$. For the numerical proof of the jump conditions in \eref{Eq:NuJ}, we refer to \cite{Zhang2015}.
\end{remark}
\subsection{Effective zero-thickness model}
It is very difficult to numerically solve $u^\omega$ in (\ref{Eq:uw})  due to the potential discontinuity across the thin insulating inhomogeneities. One way to deal with thin objects is to treat them as lower dimensional interfaces. In our case, the thin insulating objects can be considered as curves. Based on the above Theorem \ref{Th:JumpCondition}, we can describe an effective zero-thickness model by imposing the jump conditions of $[u^\omega]_k$ and $\left[ \f{\p u^\omega}{\p\nu}\right]_k$ on the curves $\mathcal L_k$.  This means that the potential $u^\om$ can be approximated by the corresponding potential $\widetilde u^{\om}$ satisfying the following effective zero-thickness model:
\begin{equation}\label{modelcell}
\left\{\begin{array}{l}
 \nabla \cdot ((\gamma^{\omega}_b+ (\gamma^{\omega}_d-\gamma^{\omega}_b)\chi_D)\nabla \widetilde u^\om)= 0 ~~\q \textrm{in}\q \Om\setminus \cup_{k=1}^{N_C}\mathcal L_k,\\
\left[\f{\p}{\p\nu}  \widetilde u^{\om}\right]_{\mathcal L_k} =0,~~~k=1,2,\ldots, N_C,\\
\left[\widetilde u^\om\right]_{\mathcal L_k}= 2\delta_k\f{1}{\lambda_c(\omega)} \f{\p \widetilde u^\omega}{\p\nu}|_+,~~~k=1,2,\ldots, N_C,\\
 \gamma^\omega_b \f{\p \widetilde u^\omega}{\p \nu} = g \quad\mbox{on}~\p \Om,
\end{array}\right.
\end{equation}
where $\chi_D$ is the characteristic function of $D$ and
 \begin{equation}\label{crack-model-jump}
  \begin{array}{lll}
[\widetilde u^\omega(x)]_{\mathcal L_k}&:=&\lim_{s\rightarrow0^+} \left(\widetilde u^\omega(x+s\nu_x )-\widetilde u^\omega(x-s\nu_x )\right) \\
 \left[ \f{\p \widetilde u^\omega}{\p\nu}(x)\right]_{\mathcal L_k}&:=& \lim_{s\rightarrow 0^+} \left(\f{\p \widetilde u^\omega}{\p\nu}(x+s\nu_x )- \f{\p \widetilde u^\omega}{\p\nu}(x-s\nu_x )\right)
 \end{array}
~\q(x\in \mathcal{L}_k).
\end{equation}
Since $u^\om \approx \tilde u^\omega$ in $\{ x\in \Om~: \mbox{dist}(x,\p\Om) < \f{d_0}{2}\}$, the forward model (\ref{Eq:uw}) and the effective zero-thickness model (\ref{modelcell}) have basically the same Neumann-to-Dirichlet data in terms of the inverse problem.
From now on, let $u^\om$ denote the solution of (\ref{modelcell}) for ease of notation. From the above zero-thickness model, the boundary condition along curve $\mathcal{L}_k$ depends on thickness $\delta_k$ of thin insulating objects as well as the value of $\lambda_c(\omega)$ which is related with injected current frequency $\omega$. From the above Theorem \ref{Th:JumpCondition}, the jump of the potential across $\mC_k$ depends on angular frequency $\omega$ as well as the thickness $\delta_k$. At high frequencies, the admittivity ratio $|\lambda_c(\omega)|$ is away from 0, and there will be no potential jump when $\delta_k\approx 0$. Whereas, at low frequencies, $|\lambda_c(\omega)|$ is close to zero since $\sigma_c/\sigma_b\approx 0$ and potential jump happens across $\mC_k$. For easy understanding, changes of potential distribution near thin insulating objects are given in Figure \ref{Fig:vjumpSimu}.
\begin{figure}[ht!]
  \centering
  \includegraphics[scale=1.2]{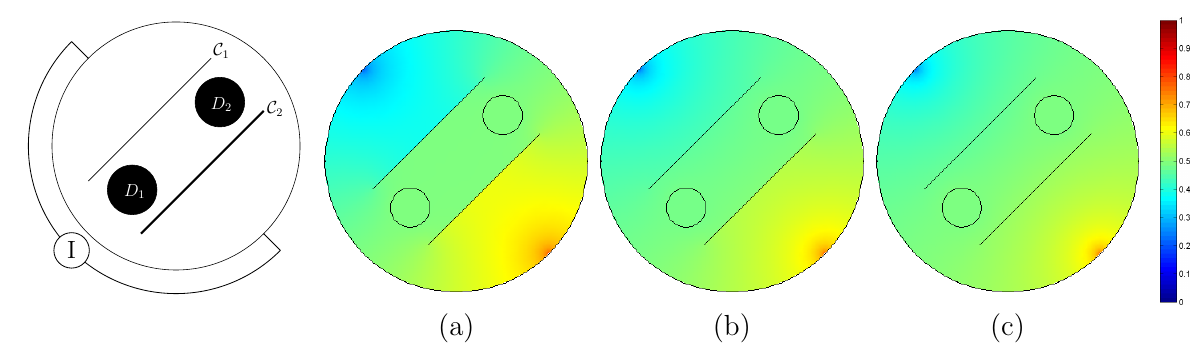}\\
  \caption{Changes of potential distribution near thin insulating objects with frequencies: (a) 10Hz, (b) 10kHz, (c) 500kHz.}\label{Fig:vjumpSimu}
\end{figure}
Based on the above analysis, we consider separately the following two cases \cite{Zribi2011}:
\begin{itemize}
 \item High-frequency case: $\delta_k\approx 0$ and $0<c_1\leq|\lambda_c(\omega)|$.
 \item Low-frequency case: $|\lambda_c(\omega)|\approx 0$ and $\delta_k\approx 0$ with $|\lambda_c(\omega)|^{-1}\delta_k \approx \beta$ and $0<\beta<\infty$.
\end{itemize}
In the upcoming sections, asymptotic formulas connecting the boundary potential and inhomogeneities are used to derive an explicit formula for identification of thin insulating objects and small conductive inhomogeneities inside $\Omega$.
\subsection{High-frequency case: \boldmath{$\delta_k\approx 0$ and $0<c_1\leq|\lambda_c(\omega)|$ }}
In high-frequency case, we suppose that the injected current frequency $\omega$ is not that low, so that $|\lambda_c(\omega)|$ is away from zero. When the thickness $\delta_k$ goes to zero, the potential jump along $\mathcal{C}_k$ also goes to zero according to approximation formula (\ref{Eq:NuJ}) as well as Figure \ref{Fig:vjumpSimu}. Therefore, the proposed problem can be regarded as traditional impedance boundary value problem and the influence of thin insulating inhomogeneities on the high-frequency current-voltage data is very weak. In this case, the following boundary voltage asymptotic expansion holds at high-frequencies. For detailed analysis and similar proof, one may refer to \cite{Ammari2006,Beretta2003,Beretta2001,Frieman1989} in which the proof can be immediately extended to the complex valued equation.

\begin{theorem}[Asymptotic expansion at high-frequencies]\label{Th:highasy} Let $\lambda_c(\omega)$ and $\delta_k$ satisfy the conditions stated in high-frequency case. Assume that $u^\om$ is a solution of the effective zero-thickness model (\ref{modelcell}) and $u_0$ is the solution of equation (\ref{Eq:uw}) with $\gamma^\omega = \gamma^\omega_b$. Assume that all $\mathcal L_k$ are line segments with endpoints $P_k,Q_k$. For $x\in\p\Omega$, when the injection current frequency is high, the perturbations of voltage potential $u^\omega$ due to small conductive objects $D_k$ and thin insulating objects $\mC_k$ can be expressed as
\begin{eqnarray}\label{Eq:SepU}
&& \left(-\f{1}{2} {I}+\mathcal{K}_\Omega\right)[u^\omega-u_0](x)\nonumber \\
&=& - \sum_{k=1}^{N_C}\int_{\mathcal{L}_k} \delta_kA_k(x',\lambda_c(\omega))\nabla u_0(x')\cdot\nabla \Gamma(x,x')ds_{x'}\nonumber \\
&&-\delta_D^2 \sum_{k=1}^{N_D}\nabla \Gamma(x,z_k)\cdot
M(\lambda_d(\omega),B_k)\nabla u_0(z_k) + O(\delta_k^2)+ O
(\delta_D^3),
\end{eqnarray}
where $A_k(x,\lambda_c(\omega))$ is a $2\times 2$ symmetric matrix whose eigenvectors are unit normal vector
 $\nu_k(x)$ and unit tangential vector $\tau_k(x)$ to $\mL_k$ and the corresponding eigenvalues are $2(1-\f{1}{\lambda_c(\omega)})$ and $2\left(\lambda_c(\omega)-1\right)$, respectively. Moreover, $M(\lambda_d(\omega),B_k)$ is polarization tensor given by
\begin{equation}\label{Eq:GPT}
M_{ij}:=\int_{\p B_k}y^j(\lambda_d(\omega)I-\mathcal{K}^*_{B_k})^{-1}(\nu_x\cdot\na x^i)(y)ds_y,\quad i,j=1,2,
\end{equation}
with \begin{equation}\label{Eq:lamD}
  \lambda_d(\omega) = \f{(\sigma_d+\sigma_b)+i\omega(\epsilon_d+\epsilon_b)}{2((\sigma_d-\sigma_b)
  -i\omega(\epsilon_d-\epsilon_b))}. \end{equation}
\end{theorem}
\begin{remark}\label{rmk:theorem1}
In Theorem \ref{Th:highasy}, the matrix $A_k$ is only related to the geometry of the thin insulators and admittivity ratio $\lambda_c(\omega)$. Since $\sigma_c$ is very small, the ratio $\lambda_c(\omega)$ changes much with respect to the operating frequency $\omega$. On the other hand, polarization tensor $M(\lambda_d(\omega), B_k)$ is related with the geometry of small conductors and admittivity distribution. When $D_k$ is a circular disk, $M(\lambda_d(\omega), B_k)$ can be explicitly written as $M(\lambda_d(\omega), B_k)=\f{|B_k|}{\lambda_d(\omega)}I$. Since $\sigma_d$ is realtively large compared with the background conductivity $\sigma_b$ and $\sigma_d-\sigma_b$ is away from zero, $\lambda_d(\omega)$ varies little with respect to frequency. Thus conductive objects are insensitive to the boundary measurements.
\end{remark}
Theorem \ref{Th:highasy} shows that the measured
boundary data are influenced by insulating objects and conductive objects since the first term on the right-hand side of formula (\ref{Eq:SepU}) is only related with thin insulators while the second term is only related with small conductors. Depending on the magnitude of $\om, \delta_k$ and $\delta_D$, the dominative term on right-hand side of formula
(\ref{Eq:SepU}) may be alternative. To see the effect of $\om,
\delta_k$ and $\delta_D$ on the measured boundary data more
clearly, we need further analysis on the expansion formula in
Theorem \ref{Th:highasy}.

To avoid any confusion, we will adopt the following notations: $x=(x_1,x_2)$ denotes a point in $\mathbb{R}^2$ and $\mathbf{x}=x\cdot(1,i)$ will be the corresponding point in ${\mathbb C}$. Let $g= a\cdot\nu$ on $\p\Omega$ where $a$ is a unit vector in $\mathbb{R}^2$. Similarly, $ {\bf a}=a\cdot(1,i)$, the center $z_k$ of each small conductor $D_k$ can be expressed as ${\z_k} = z_k\cdot(1,i)$, the endpoints of line segment $\mL_k$ can be written as ${\P_k}=P_k\cdot(1,i), ~\hbox{and} ~~{\Q_k}=Q_k\cdot(1,i).
$  Then we have the following result.
\begin{theorem}[Identification of thin insulators and small conductors]\label{Th:rebarC}
Let $\lambda_c(\omega)$ and $\delta_k$ satisfy the conditions stated in high-frequency case. Assume that all the thin inhomogeneities are line segments and all the small conductive objects are disks. Then $(-\f{1}{2}I+\mathcal{K}_\Omega)[u^\omega-u_0]$ on the boundary $\p\Om$ can be expressed as
\begin{eqnarray}\label{Eq:holo}
\Re\left\{ (-\f{1}{2}I+\mathcal{K}_\Omega)[u^\omega-u_0]({\bf x})\right\} = \Re\{G^\Re(\x)\}
+O (\delta_k^2)+ O (\delta_D^3),\\
\Im \left\{(-\f{1}{2} {I}+\mathcal{K}_\Omega)[u^\omega-u_0]({\bf
x}) \right\}=\Re\left\{G^\Im(x)\right\} +O (\delta_k^2)+ O (\delta_D^3),
\end{eqnarray} where $G^\Re$ and $G^\Im$ are meromorphic  functions:
\begin{eqnarray}
\hspace{-2cm}\f{dG^\Re(\x)}{d\x}=\sum_{k=1}^{N_C}
{\mathfrak{C}_k^{\Re}(\om,\delta_k)}\left(\f{1}{\x-\Q_k}-\f{1}{\x-\P_k}\right)-
\sum_{k=1}^{N_D}{\mathfrak{D}_k^{\Re}(\om,\delta_D)}\f{1}{(\x-\z_k)^2}\label{Eq:holo2-1}\\
\hspace{-2cm}\f{d G^\Im(\x)}{d \x} = \sum_{k=1}^{N_C}{\mathfrak{C}_k^{\Im}(\om,\delta_k)}\left(\f{1}{\x-\Q_k}-\f{1}{\x-\P_k}\right)
-\sum_{k=1}^{N_D}{\mathfrak{D}_k^{\Im}(\om,\delta_D)}\f{1}{(\x-\z_k)^2}
 \label{Eq:holo2-2}\end{eqnarray}
and
 \begin{eqnarray}
\mathfrak{C}_k^{\Re}(\omega,\delta_k)&=\f{\delta_k}{\pi}\left(\Re\{(\lambda_c(\omega)-1)\}a_{\tau_k}+i\Re\{(1-\f{1}{\lambda_c(\omega)}) \} a_{\nu_k}\right)\label{Def:holo3-1}\\
\mathfrak{C}_k^{\Im}(\omega,\delta_k)&=\f{\delta_k}{\pi}\left(\Im\{\lambda_c(\omega)-1\}a_{\tau_k}+i\Im\{1-\f{1}{\lambda_c(\omega)}\} a_{\nu_k}\right)\label{Def:holo3-2}\\
{\mathfrak{D}}_k^{\Re}(\omega,\delta_D)&=-\Re\left\{\f{|B_k|\delta_D^2}{2\pi\lambda_d(\omega)}\right\}\a,
\quad \mathfrak{D}_k^{\Im}(\omega,\delta_D)
=-\Im\left\{\f{|B_k|\delta_D^2}{2\pi\lambda_d(\omega)}\right\}\a
.\label{Def:holo3-3}
 \end{eqnarray}
Here, $a_{\nu_k}=a\cdot\nu_k, a_{\tau_k}=a\cdot\tau_k$.
\end{theorem}
\begin{proof}
Since $B_k$ is a disk, the formula (\ref{Eq:GPT}) gives
$M(\lambda_d(\omega),B_k)=\f{|B_k|}{\lambda_d(\omega)}I$.  Hence,
the formula (\ref{Eq:SepU}) in Theorem \ref{Th:highasy} can be
expressed as
\begin{equation}\label{Phi0}
 \left(\f{1}{2} {I}+\mathcal{K}_\Omega\right)[u^\omega-u_0](x)=\Phi(x) + O (\delta_k^2)+
 O (\delta_D^3)\q \q (x\in\p\Om),
\end{equation}
where $\Phi$ is
\begin{equation}\label{Phi1}
\hspace{-2cm}\Phi(x)= -
\sum_{k=1}^{N_C}\delta_k\int_{\mathcal{L}_k}(A_k~a)\cdot\nabla
\Gamma(x,x')ds_{X'}-\f{\delta_D^2}{2\pi}
\sum_{k=1}^{N_D}\f{|B_k|}{\lambda_d(\omega)}\f{x-z_k}{|x-z_k|^2}\cdot
a .
\end{equation}
We use $a=a_{\nu_k}\nu_k+ a_{\tau_k}\tau_k$ to get
\begin{eqnarray}
\hspace{-2cm}\Phi(x)&=&-\f{1}{2\pi}\sum_{k=1}^{N_C}\delta_k\int_{\mathcal{L}_k}\left(2(1-\f{1}{\lambda_c(\omega)})
a_{\nu_k}\nu_k+ 2(\lambda_c(\omega)-1)a_{\tau_k}\tau_k\right)\cdot\f{x-x'}{|x-x'|^2}ds_{x'}\nonumber\\
\hspace{-2cm}&&-\f{\delta_D^2}{2\pi\lambda_d(\omega)}\sum_{k=1}^{N_D}|B_k|\f{x-z_k}{|x-z_k|^2}\cdot
a \q\q\q\q(x\in\p\Om).
 \label{Phi2}
\end{eqnarray}

 From now on, we shall identify $\mathbb{R}^2$ with the complex plane $\mathbb{C}$ and use similar ideas as those in \cite{Ammari2006,Baratchart1999}. Since $\lambda_c(\omega)$, $\lambda_d(\om)$ as well as $u^\omega$ are complex, we will consider real and imaginary parts of $\Phi(x)$ separately.

The real part of $\Phi(x)$ for $x\in\p\Om$  can be expressed as
\begin{eqnarray}\label{Eq:maincom}
\hspace{-2cm}\Re\{\Phi(\x)\}=\Re\left\{
-\f{1}{2\pi}\sum_{k=1}^{N_C}\delta_k\int_{\mathcal{L}_k}{
\f{\mbox{\boldmath $\xi$}}{\bf x- \bf x'}}d s_{\bf
x'}-\Re\left\{\f{\delta_D^2}{2\pi\lambda_d(\omega)}
\right\}\sum_{k=1}^{N_D}|B_k| \f{\bf a}{\x-\z_k}\right\},
\end{eqnarray}
where \mbox{\boldmath $\xi$}$=\xi\cdot (1,i)$ and $\xi$ is
\begin{eqnarray}\label{xi}
\xi=\Re\{2(1-\f{1}{\lambda_c(\omega)})a_{\nu_k}\nu_k+ 2(\lambda_c(\omega)-1)a_{\tau_k}\tau_k\}.
\end{eqnarray}
Since  $\mathcal L_k$ is the segment with endpoints $P_k, Q_k$,  it can be written as $P_k+t(Q_k-P_k), 0\leq t\leq 1$. Therefore, $\mathcal L_k$ has its unit tangent vector ${\mbox{\boldmath $\tau$}_k} ={\f{\Q_k-\P_k}{|\P_k-\Q_k|}}$ and its unit normal vector ${\mbox{\boldmath $\nu$}_k} = i{ \f{\Q_k-\P_k}{|\P_k-\Q_k|}}$ in $\mathbb{C}$. Hence, the integral term in (\ref{Eq:maincom}) can be written as
\begin{eqnarray*}
  \int_{\mathcal{L}_k}{\bf \f{\mbox{\boldmath $\xi$}}{x-x'}}d s_{\bf x'} &=& {|\Q_k-\P_k|} \int_0^1 \f{\mbox{\boldmath $\xi$}}{{ (\x-\P_k)}-t{ (\Q_k-\P_k)}} dt\\
  &=& \f{ \mbox{\boldmath $\xi$}|\Q_k-\P_k|}{\Q_k-\P_k}\ln{\f{\x-\P_k}{\x-\Q_k}}.
\end{eqnarray*}
From (\ref{xi}), we have
\begin{eqnarray*}
 \hspace{-2cm}&& \f{\mbox{\boldmath $\xi$}|\Q_k-\P_k|}{\Q_k-\P_k}\nonumber\\
 \hspace{-2cm}&&=\f{ |\Q_k-\P_k|}{\Q_k-\P_k}\left(2\Re\{1-\f{1}{\lambda_c(\omega)}\} a_{\nu_k}{\f{i(\Q_k-\P_k)}{|\P_k-\Q_k|}}+2\Re\{\lambda_c(\omega)-1\}a_{\tau_k}{\bf \f{Q_k-P_k}{|P_k-Q_k|}}\right)\nonumber\\
\hspace{-2cm}&&\q= 2\Re\{\lambda_c(\omega)-1\}a_{\tau_k}+i2\Re\{1-\f{1}{\lambda_c(\omega)}\} a_{\nu_k}.
\end{eqnarray*}
Therefore, (\ref{Eq:maincom}) can  be simplified as
\begin{eqnarray}\label{Phi3}
\Re\{\Phi({\bf
x})\}=\Re\left\{\sum_{k=1}^{N_C}{\mathfrak{C}_k^{\Re}(\omega,\delta_k)}\ln{\f{\x-\Q_k}{\x-\P_k}}+
\sum_{k=1}^{N_D}{\mathfrak{D}_k^{\Re}(\omega,\delta_D)}\f{1}{\x-\z_k}\right\},
\end{eqnarray}
where $\mathfrak{C}_k^{\Re}(\omega,\delta_k)$ and $\mathfrak{D}_k^{\Re}(\omega,\delta_D)$ are the quantities defined in (\ref{Def:holo3-1}) and (\ref{Def:holo3-3}).
From (\ref{Phi3}), the real part of $\Phi$ can be viewed as the real part of the meromorphic function $G^{\Re}(\x)$ given by
\begin{eqnarray}\label{Gre}
G^{\Re}(\x):=\sum_{k=1}^{N_C}{\mathfrak{C}_k^{\Re}(\omega,\delta_k)}\ln{\f{\x-\Q_k}{\x-\P_k}}
+\sum_{k=1}^{N_D}{\mathfrak{D}_k^{\Re}(\omega,\delta_D)}\f{1}{\x-\z_k}.
\end{eqnarray}
Since $G^{\Re}(\x)$ is holomorphic except at the points $\P_k, \Q_k, \z_k$ and the segments $\P_k\Q_k$ (see, for instance, \cite{Stein2003}), it has complex derivative near $\p\Om$ in the complex plane:
\begin{eqnarray}\label{Gre2}
\hspace{-2cm}\f{d G^{\Re}(\x)}{d\x} =
\sum_{k=1}^{N_C}{\mathfrak{C}_k^{\Re}(\omega,
\delta_k)}\left(\f{1}{\x-\Q_k}-\f{1}{\x-\P_k}\right)-\sum_{k=1}^{N_D}{\mathfrak{D}_k^{\Re}
(\omega,\delta_D)}\f{1}{(\x-\z_k)^2}.
\end{eqnarray}
A similar argument applies for the imaginary part of $\Phi(\x)$.
\end{proof}

The followings remark on Theorem \ref{Th:rebarC} is in due.
\begin{remark}
According to Theorem \ref{Th:rebarC}, both $G^{\Re}(\x)$ and
$G^{\Im}(\x)$ can be viewed as known quantities  from the
knowledge of $(-\f{1}{2}I+\mathcal{K}_\Omega)[u^\omega-u_0]$ on
$\p\Om$. This theorem states that $\f{d G^{\Re}({\bf x})}{d {\bf
x}}$ is a meromorphic function in $\mathbb{C}$ with simple poles
at the endpoints $\P_k,~\Q_k$ of the segments $\mathcal{L}_k$ and
poles of order $2$ at the center $\z_k$ of $D_k$.  Hence, the
residues of $\f{d G^{\Re}({\bf x})}{d {\bf x}}$ at the endpoints
are given by
\begin{eqnarray}\label{residue}
\mbox{Res} \left(\f{d G^{\Re}({\bf x})}{d {\bf x}},{\bf Q_k}\right)= \mathfrak{C}_k^{\Re}(\omega,\delta_k) = -
 \mbox{Res} \left(\f{d G^{\Re}({\bf x})}{d {\bf x}},{\bf P_k}\right).
\end{eqnarray}
The information of the center of $D_k$ is contained in the following function
\begin{eqnarray}\label{Eq:w}
w({\bf x})&:=&\sum_{k=1}^{N_D}{\mathfrak{D}_k^{\Re}}(\om,\delta_D)\f{1}{(\x-\z_k)^2}.
\end{eqnarray}
Then the function $\f{w'({\bf x})}{w({\bf x})}$ will have simple poles at the poles of $w(\x)$. Hence, these center points can be identified from boundary measurements \cite{Kang2004}.
\end{remark}

The above analysis shows that the effect of thin insulators on the boundary data highly depends on the frequency, while the effect of small conductive objects does not depend on the operating frequency that much. This relation leads to the assertion that we can detect the small conductors when the frequency is  very high and  both thin insulators and small conductors when the frequency decreases.  Numerical simulations in the later part of this paper will illustrate this important observation.
We can similarly analyze $\mathfrak{C}_k^{\Im}(\omega,\delta_k)$ and $\mathfrak{D}_k^{\Im}(\omega,\delta_D)$. For the low frequency case, we have the following results.
\subsection{Low-frequency case: {\boldmath $|\lambda_c(\omega)|\approx 0$ and $\delta_k\approx 0$ with $|\lambda_c(\omega)|^{-1}\delta_k \approx \beta$ and $0<\beta<\infty$}}

In the low-frequency case,  the admittivity contrast $\lambda_c(\omega)$ is getting close to zero. As thickness $\delta_k$ goes to zero, we suppose that
\begin{equation} \label{assumpopt}
\f{1}{|\lambda_c(\omega)|}\delta_k\approx\beta.
\end{equation}
 Then by approximation \eref{Eq:NuJ}, the potential jump along each thin insulator can not be ignored. According to \cite{Ammari2006,Ammari2003,Frieman1989,Zribi2011}, we have the following asymptotic expansion formula of the potential $u^{\om}$ for low frequency current.
\begin{theorem}[Asymptotic expansion at low frequencies]
In the low-frequency regime, we have the following asymptotic formula for the boundary perturbations of the potential $u^\omega$:
\begin{eqnarray}\label{Eq:lowAsympt}
 \hspace{-2cm}&\indent \displaystyle \left(-\f{1}{2}I+K_\Omega\right)[u^\omega-u_0](x) =-\delta_D^2 \sum_{k=1}^{N_D}\nabla \Gamma(x,z_k)\cdot M(\lambda_d,B_k)\nabla u_0(z_k)\nonumber\\
\hspace{-2cm}&  \indent \displaystyle +\sum_{k=1}^{N_C} \int_{\mathcal{L}_k}\f{\p \Gamma(x,x')}{\p \nu_{x'}} [u^\omega]_k(x')dx' +
O (\delta^2_k) +  O (\delta_D^3).
\end{eqnarray}
\end{theorem}
In this case, since the potential jump $[u^\omega]_k(x)=\f{2\delta_k}{\lambda_c(\omega)} \f{\p u^\omega}{\p\nu}(x-\delta_k\nu_x )|_+$ along $\mathcal L_k$ is very large and could not be ignored, the effect of small conductors on the perturbations of the boundary voltage is hidden by the insulators. Although we cannot write (\ref{Eq:lowAsympt}) in an explicit way, we know that it is related with the endpoints as well as the potential jump along the thin insulators. When multiple thin insulators are well separated from each other, we can always image them from boundary measurements. However, small conductors at low frequencies are invisible since the thin insulators will dominate the boundary measurements. Note that
(\ref{assumpopt}) gives the optimal range of operating frequencies to use for optimal imaging of the thin insulators.

\subsection{Spectroscopic analysis}
Based on the above analysis in low- and high-frequency regimes, we mathematically derived the frequency dependency of the current-voltage data in a rigorous way. The current-voltage data is mainly affected by the outermost thin insulators when the frequency is low, whereas the data mainly depends on the small conductors when the frequency is high. With this observation, we can detect the outermost thin insulators at low frequencies. As the frequency increases, the small conductors become gradually visible whereas thin insulators fade out (thicker insulator fades out at higher frequency than thinner one). Hence, multi-frequency EIT system allows to probe these frequency dependent behavior. On the other hand, we could not give a clear range of the frequency to best identify thin objects or small conductors since the penetrating frequency is not only depending on thickness $\delta_k$ but also relying on admittivity contrast between $\gamma_c$ and $\gamma_b$, geometry, relative size between $\delta_k$ and domain $\Omega$, position distribution, etc. We plan to deal with this challenging issue in the future.
\section{Numerical simulations}
In this section, we will verify our mathematical analysis through various numerical simulations by using multi-frequency Electrical Impedance Tomography (mfEIT). The rough procedure of mfEIT using the standard sensitivity matrix \cite{Holder2005,Seo2012} is as follows:
\begin{enumerate}
\item $N_E$ electrodes $\mE_1,\mE_2,\cdots,\mE_{N_E}$ are  attached on the boundary of $\Omega$ with a unform distance between adjacent electrodes. Inject current between all adjacent pair of electrodes $\mE_k,\mE_{k+1}$ at various angular frequencies ($\omega_1,\om_2,\cdots, \om_{N_\om}$).
\item Solve the forward problem  using finite element method (FEM) to compute the potential $u^\om_k$ due to $k$-th injection current and collect simulated boundary voltage data $V_{\omega}=(V_\om^{1,1}, \cdots, V_\om^{N_E,N_E})$, where $V^{k,j}$ is the $j$th boundary voltage subject to the $k$th injection current, which is given by $
V^{j,k}_\om=\int_\Om \gamma^{\om}\na u_k^{\om}\cdot \na u_j^{\om} dx$.  (In real experiment, we do not use   $V^{k-1,k}_\om, V^{k,k}_\om, V^{k,k+1}_\om$ due to the unknown contact impedance.)
\item Discretizing the  domain $\Om$ into $M$ elements or pixels as $\Omega = \sum_{k=1}^M \mT_k$, compute the standard  sensitivity matrix (see \cite{Holder2005,Seo2012}). We use the standard linearized method to reconstruct the  admittivity images $ \delta\gamma_\om= \{\delta\gamma^{\omega}_{\mT_1},\delta\gamma^{\omega}_{\mT_2},\cdots,\delta\gamma^{\omega}_{\mT_M}\}^T$ from the boundary data $V_{\omega_j}$. Here, the reference homogeneous admittivity was subtracted.
\end{enumerate}

The numerical simulations are performed on a unit disk $\Omega=\{(x,y):x^2+y^2<1\}$ with 16 electrodes equally spaced around its circumference. Inside the unit disk, we consider the following numerical models with two conductive objects and several thin insulating objects inside as shown in Figure \ref{Fig:SimModel}.
 \begin{figure}[ht!]
  \centering
  \begin{tikzpicture}
  \node at (0,0) {\includegraphics[scale=0.7]{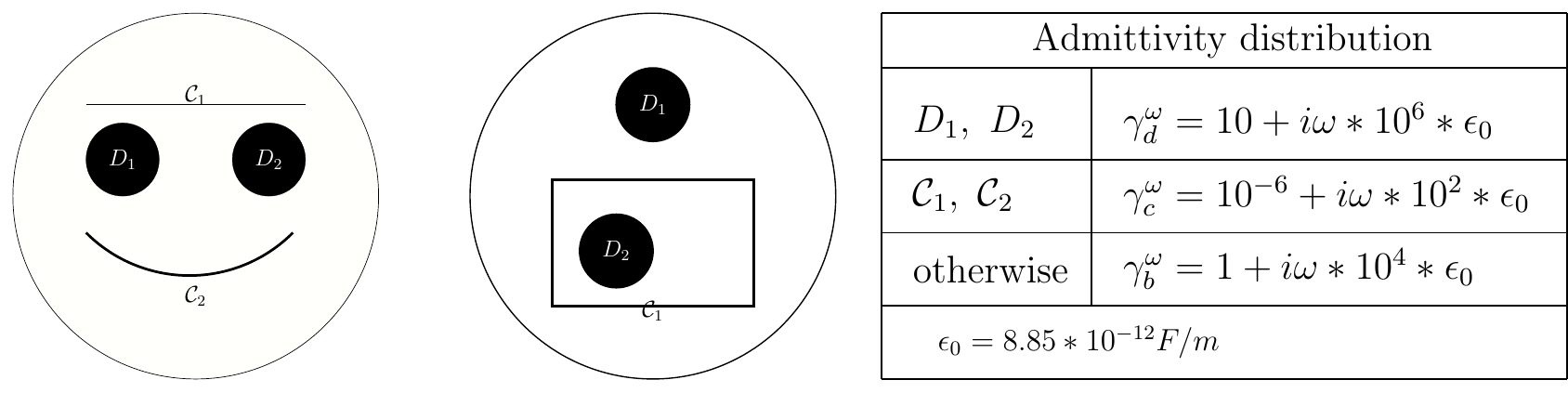}};
  \node at (-4.5,-2) {(a)};
   \node at (-0.8,-2) {(b)};
   \node at (3,-2) {(c)};
  \end{tikzpicture}
  \caption{Two different numerical models: (a) Two thin insulating inhomogeneities $\mC_1,~\mC_2$ with same thickness and two small conductive objects and (b) two small conductive objects $D_1$ and $D_2$, and a rectangular insulating inhomogeneity encircling $D_2$; (c) Admittivity distribution for thin insulating inhomogeneities $\mC_1,~\mC_2$ and small conductive objects $D_1,~D_2$.}\label{Fig:SimModel}
\end{figure}
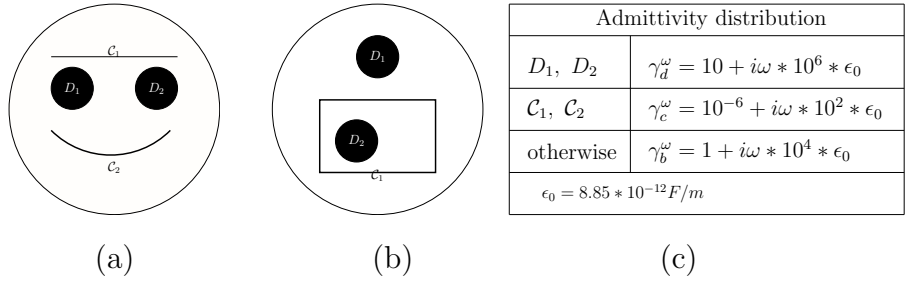
 In Figure \ref{Fig:SimModel} (a), two thin insulating inhomogeneities with the same thickness and two conductive objects are placed inside a homogenous domain so that we can investigate the effect of current frequencies on the reconstructed images. Thereafter, in Figure \ref{Fig:SimModel} (b), we will take use of a simulation model with two conductive objects encircled by a rectangular insulating object $\mC_1$. The thickness of the three thin insulating objects is $5*10^{-4}$ in Figure \ref{Fig:SimModel} (a, b). The admittivity distribution is a piecewise constant in each subdomain as shown in Figure \ref{Fig:SimModel} (c).   A wide range of current frequencies is applied and the resulting boundary potential  are collected. The selected spectroscopic images of admittivity distribution are shown in Figure \ref{Fig:RecomSimModel} where the admittivity distribution is reconstructed at frequencies $\omega=$ 10Hz, 1kHz, 50kHz, 150kHz, 250kHz and 500kHz.

\subsection{Multi-frequency images}
 Figure \ref{Fig:RecomSimModel} shows the reconstructed spectroscopic images when the simulation model is Figure \ref{Fig:SimModel} (a); When current is injected at low frequencies ($\f{\omega}{2\pi}=$ 10Hz, 1kHz), two small conductive objects are invisible since the low frequency currents are blocked by the thin insulating inhomogeneities and the boundary potential is mostly influenced by the thin insulating objects.
 \begin{figure}[ht!]
  \centering
  \includegraphics[scale=0.7]{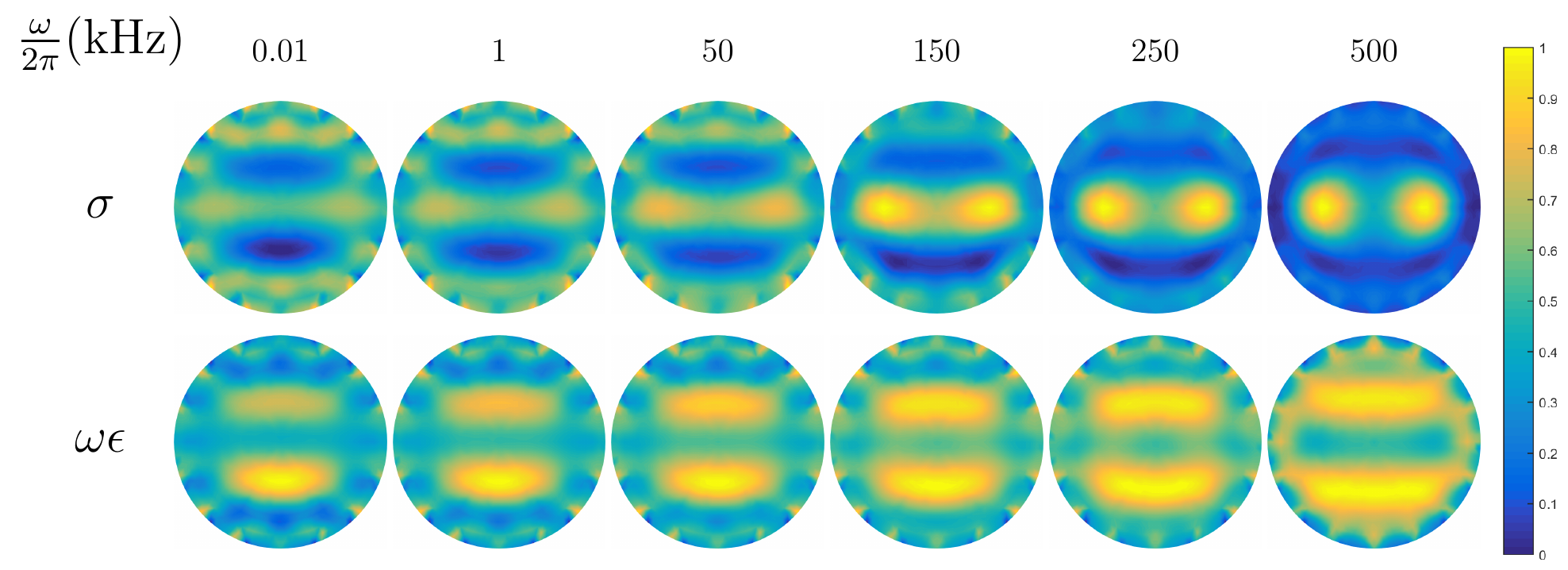}\\
  \caption{Reconstructed admittivity distribution when current is injected at frequencies $\omega=$ 10Hz, 1kHz, 50kHz, 150kHz, 250kHz, 500kHz. The
second row contains reconstructed images for normalized conductivity $\sigma$ (S/m), and the third row is reconstructed images for normalized $\omega\epsilon$ (S/m). }\label{Fig:RecomSimModel}
\end{figure}
As frequency increases, the currents at frequencies 50kHz and 150kHz begin to penetrate the thin insulating objects and conductive objects start appearing in the reconstructed images. Finally, when the currents are applied at frequencies 250kHz and 500kHz, both conductive objects are visible in the reconstructed images because the potential is mainly affected by the conductive objects $D_1,~D_2$.
 \begin{figure}[ht!]
  \centering
  \includegraphics[scale=1]{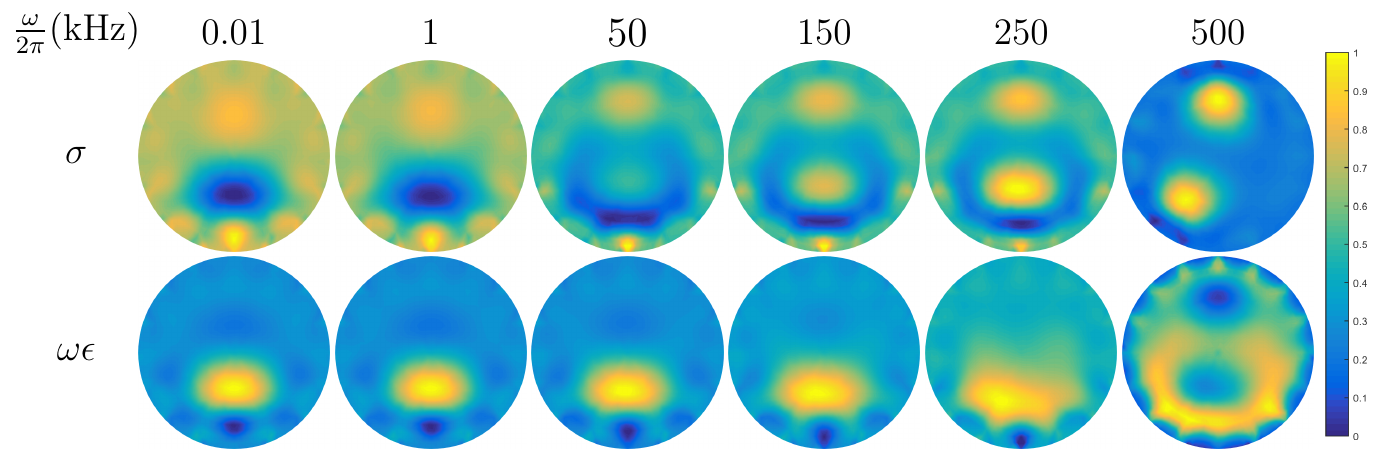}\\
  \caption{Reconstructed admittivity distribution when current is injected at frequencies $\omega=$ 10Hz, 1kHz, 50kHz, 150kHz, 250kHz, 500kHz. The
second row contains reconstructed images for normalized conductivity $\sigma$ (S/m), and the third row is reconstructed images for normalized $\omega\epsilon$ (S/m). }\label{Fig:RecomSimModel2}
\end{figure}

As to the second numerical model in Figure \ref{Fig:SimModel} (b), we illustrate the reconstructed spectroscopic images in Figure \ref{Fig:RecomSimModel2}. Similarly, conductive object surrounded by the rectangular insulating object is not visible at low frequencies 10kHz and 1kHz because the currents can not pass through the thin insulating objects. On the other hand, when frequency increases from 50kHz to 250kHz, the insulating rectangular 'wall' starts to fade out because the current begins to penetrate insulating object and the boundary voltage is influenced by the internal conductive object. When the frequency is very high (500kHz), the insulating rectangular object $\mC_1$ disappears and only conductive objects are visible in the reconstructed image.  It is due to the fact that the data are mainly influenced by conductive objects. From the point of view of inverse problem, the spectroscopic images can give the spectroscopic information for the thin insulating objects $\mC_1,\mC_2$ by providing how much the boundary potential is affected by thin insulating inhomogeneities at various frequencies. The influence of thin insulating objects on the boundary measurements decreases as the frequency increase. The dominant influence at low frequencies is from insulating objects while the boundary potential is dominated by conductive objects at high frequencies.
 \subsection{Fusion of multi-frequency images}
 Now we are considering to construct an integrated image based on an investigation of principal component analysis (PCA) where the information of both thin insulating objects and small conductive objects can be extracted from the integrated image. PCA is a useful tool in extracting the dominant features (principal components) from a set of reconstructed images at various frequencies \cite{Sophian2003,Turk1991}. In order to implement this approach, we first obtain the reconstructed images at a broad range of frequencies $\omega=\omega_1,\omega_2,\cdots,\omega_{N_\omega}$ and each reconstructed admittivity. Then we represent each image by an $M\times 1$ column vector ${\bm\delta}{\bm\gamma_{\omega_j}}, j=1,2,\cdots, N_\omega$ where ${\bm\delta}{\bm\gamma_{\omega_j}} = \{\delta\gamma^{\omega_j}_{\mT_1},\delta\gamma^{\omega_j}_{\mT_2},\cdots,
 \delta\gamma^{\omega_j}_{\mT_M}\}^T$. Then average admittivity  at $\om$  can be defined as
\begin{equation}\label{Def:avgadmt}
 \overline{{\bm\delta}{\bm\gamma_{\omega}}} = \f{1}{N_\omega} \sum_{j=1}^{N_\omega} {\bm\delta}{\bm\gamma_{\omega_j}}.
\end{equation}
We define the mean oscillation of ${\bm\delta}{\bm\gamma_{\omega}}$ by  the vector
 \begin{equation}\label{Def:diffavgadmt}
\hat{{\bm\delta}{\bm\gamma_j}} = {\bm\delta}{\bm\gamma_{\omega_j}} - \overline{{\bm\delta}{\bm\gamma_{\omega}}}.
  \end{equation}
Since the real part of admittivity is $\sigma$ and its imaginary part is $\omega\epsilon$, we resolve the set of admittivities and deal with the real and imaginary part separately.  

Let $ \hat{{\bm\delta}{\bm\sigma_j}} $ be the real part of $\hat{{\bm\delta}{\bm\gamma_j}}$. This set of vectors is then subject to principal component analysis which tries to find a set of $N_\omega$ orthogonal vectors and their associated eigenvalues. Both information can best describe the distribution of the admittivity. The vectors and scalars are the eigenvectors and eigenvalues of the covariance matrix
\begin{equation}\label{Eq:eigencovariance}
 \mathbf{C_\sigma} = \f{1}{N_\omega}\sum_{j=1}^{N_\omega}\hat{{\bm\delta}{\bm\sigma_j}} ~\hat{{\bm\delta}{\bm\sigma_j}} ^T= \bm{\hat{\delta\sigma}~\hat{\delta\sigma}^T},
\end{equation}
where $\bm{\hat{\delta\sigma}}=[\hat{{\bm\delta}{\bm\sigma_1}}~ \hat{{\bm\delta}{\bm\sigma_2}} ~ \hat{{\bm\delta}{\bm\sigma_3}}\cdots \hat{{\bm\delta}{\bm\sigma_{N_\omega}}}]$ is of size $M\times N_\omega$ and the covariance matrix $\mathbf{C_\sigma}$ is of size $M\times M$ with $M$ the pixel number of the admittivity distribution. Using the singular value decomposition, the covariance matrix $\mathbf{C_\sigma}$ can be written as
\begin{equation}\label{Eq:svdC}
 \mathbf{C_\sigma} = \bm{\hat{\delta\sigma}~\hat{\delta\sigma}^T} = {\bm U}{\bm \Lambda}{\bm U^T} = \sum_{i=1}^{M}\lambda_i \mathbf{u}_i\mathbf{u}_i^T,
\end{equation} where $\lambda_i$ is the eigenvalue of $\bm{\hat{\delta\sigma}\hat{\delta\sigma}^T}$ with $\lambda_1\geq \lambda_2,\cdots\geq \lambda_M$ and $\mathbf{u}_i$ is the corresponding eigenvector.
Note that in practice, the number of frequencies ($N_\omega$) is much smaller than the pixel number ($M$) of the reconstructed images. Therefore, we consider $N (< N_\om) $ leading eigenvectors of the covariance matrix that are chosen as those with the largest associated eigenvalues,  where $N $ depends on  the signal-to-noise ratio (SNR). Then the principal components for any admittivity image $\bm{\delta}\bm{\sigma}_{\omega_j},j=1,\cdots,N_\omega$, can be written as
    \begin{equation}\label{Eq:pcasigma}
  p_i = {\bf u}_i^T (\bm{\delta}\bm{\sigma}_{\omega_j}- \overline{{\bm\delta}{\bm\sigma_{\omega}}}).
\end{equation}
for $i=1,2,\cdots, N$. Here $p_i$ represents the data projected into the $N-$dimensional space of eigenvectors. In order to enhance the image details, we rewrite $\bm{\hat{\delta\sigma}}$ in terms of singular value decomposition
\begin{equation}
  \bm{\hat{\delta\sigma}} = \sum_{i=1}^{N_\omega} \sqrt{\lambda_i} {\bm u_i}{\bm v_i}^T,
\end{equation} where ${\bm v_i}$ is the eigenvectors of $\mathbf{C_\sigma}^T$ and $ {\bm u_i}=\bm{\hat{\delta\sigma}}{\bm v_i}$. We construct the following matrix from the principal components of $\bm{\delta}\bm{\sigma}_{\omega_j}$
\begin{equation}
\tilde{\bm{\delta}\bm{\sigma_{\omega}}} = \sum_{i=1}^N {\bm u_i}{\bm v_i}^T,
\end{equation}
where $\tilde{\bm{\delta}\bm{\sigma_{\omega}}}$ is a matrix of size $M\times N_\omega$. The new admittivity distribution can be obtained by
  \begin{equation}\label{Eq:avgimg}
  {\bm{\delta}\bm{\sigma}_{\omega}} = \f{1}{N_\omega}\sum_{j=1}^{N_\omega} \tilde{\bm{\delta}\bm{\sigma}_{\omega,j}},§
\end{equation} where $\tilde{\bm{\delta}\bm{\sigma}_{\omega,j}}$ is the $j-$th column of $\tilde{\bm{\delta}\bm{\sigma_{\omega}}}$. Similarly, we can obtain the imaginary part of integrated admittivity distribution. In our case, only $N=2$ eigenvectors corresponding to the two largest eigenvalues are chosen to obtain the new images. The images corresponding the configurations in Figure \ref{Fig:SimModel} (a,b) are shown respectively in Figures \ref{Fig:int1} and \ref{Fig:int2}. The integrated images show both on the thin insulating and the small conductive objects. 
\begin{figure}[ht!]
\centering
\begin{tikzpicture}
  \node at (0,0) {\includegraphics[scale=0.8]{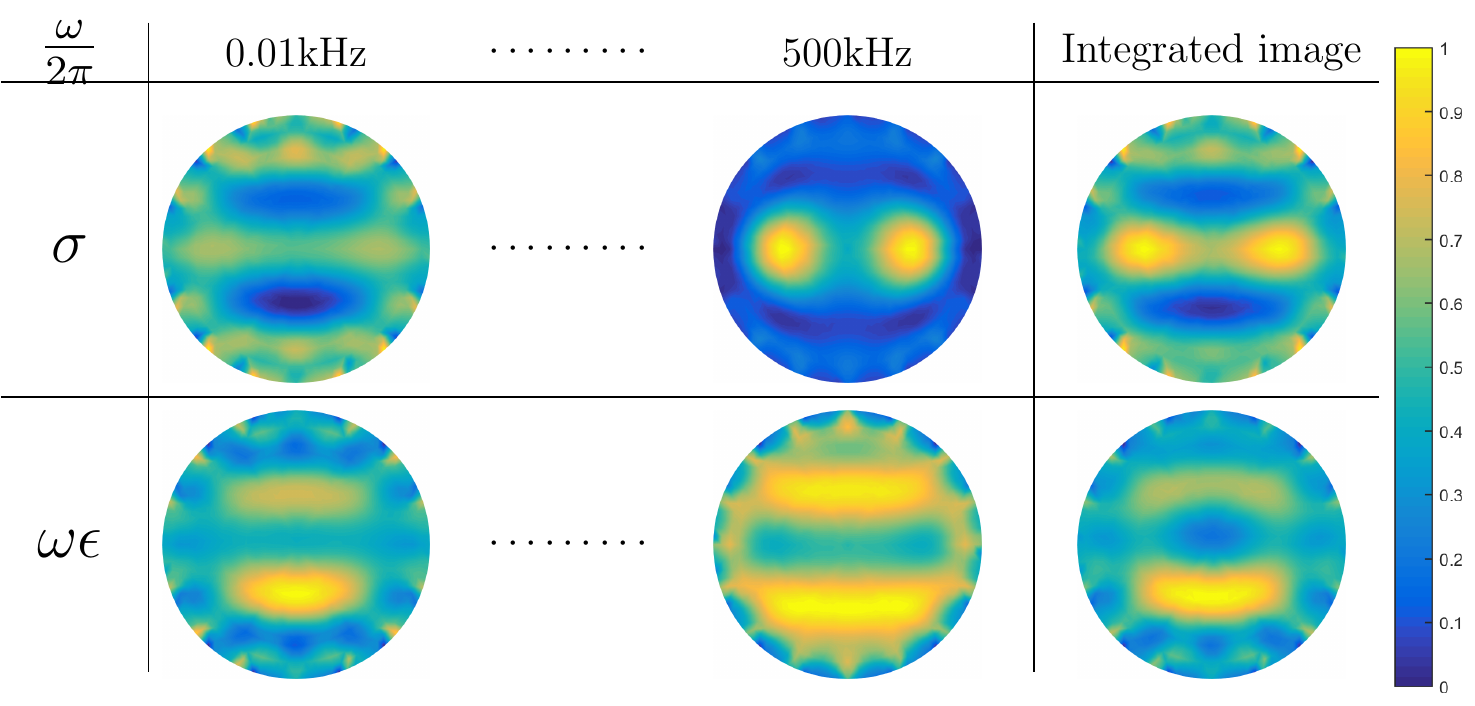}};
\end{tikzpicture}
\caption{Integrated image of real and imaginary part of normalized admittivity distribution for numerical model in Figure \ref{Fig:SimModel} (a).}\label{Fig:int1}
\end{figure}
\begin{figure}[ht!]
\centering
\begin{tikzpicture}
  \node at (0,0) {\includegraphics[scale=0.8]{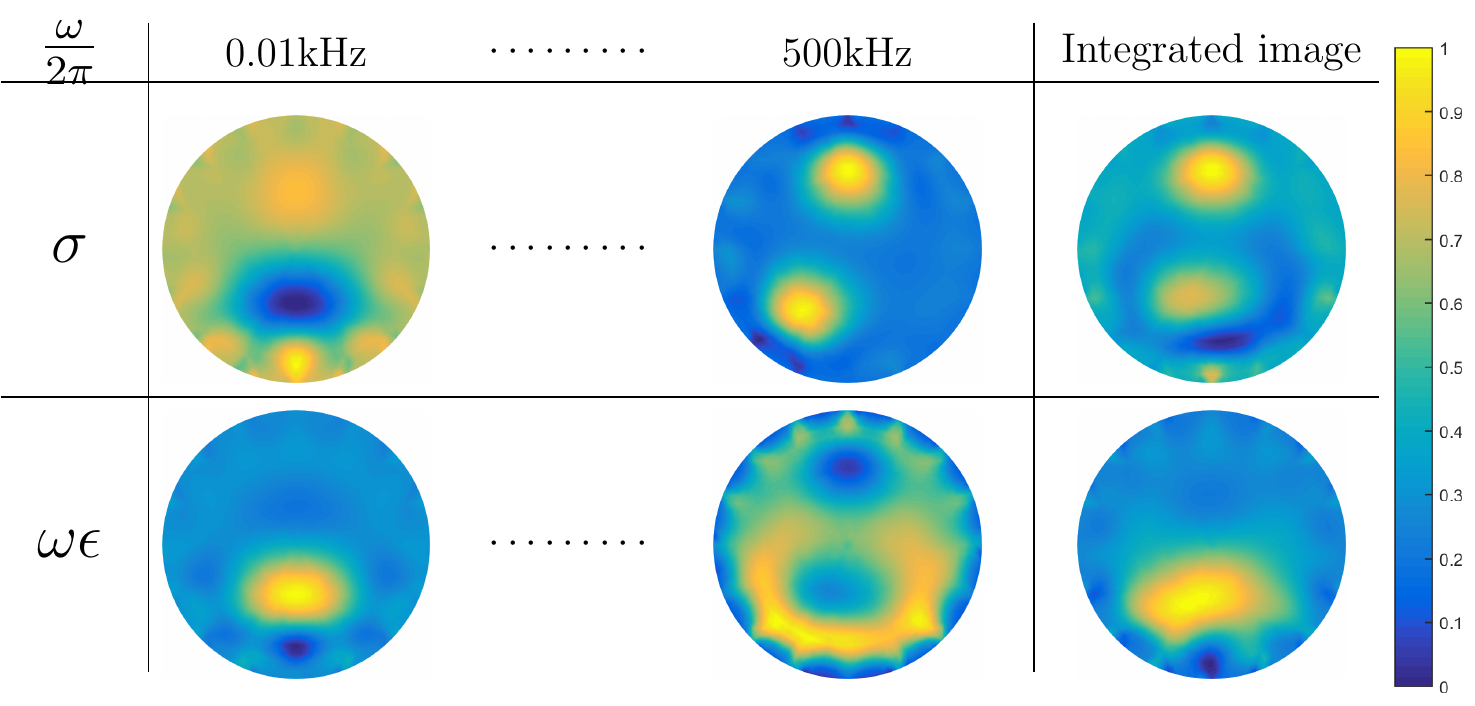}};
\end{tikzpicture}
\caption{Integrated image of real and imaginary part of normalized admittivity distribution for numerical model in Figure \ref{Fig:SimModel} (b).}\label{Fig:int2}
\end{figure}

\section{Phantom experiments}
In this section, we present  phantom experiments by using  32-channel mfEIT system (EIT-Pioneer Set made by Swisstom, Switzerland) to illustrate the frequency dependent behavior of the reconstructed images. The available injection current frequency of the Swisstom EIT-Pioneer Set is set between 50 kHz and 250 kHz.
\begin{figure}[ht!]
\centering
\begin{tikzpicture}
  \node at (0,0) {\includegraphics[scale=0.7] {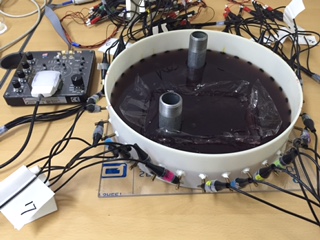}};
\end{tikzpicture}
\caption{Configuration of phantom experiments by using Swisstom EIT-Pioneer Set.}\label{Fig:ExpConfigu}
\end{figure}

Figure \ref{Fig:ExpConfigu} shows the configuration of phantom experiment. We use a cylindrical tank with 360 mm in diameter  and 32 equally-spaced electrodes are attached. The tank is filled with agar-gelatin mixture. Inside the phantom, there are two conductive objects with a diameter of 30 mm and four very thin kitchen wrap with 1 $\mu$m thickness. One conductive object is encircled by four very thin kitchen wrap. We injected current of 1 mA at  various frequencies of 50 kHz, 100 kHz, 150 kHz, 180kHz, 200 kHz and 250 kHz. Figure \ref{Fig:exp} (a) presents the reconstructed images at six different frequencies. The insulating wraps appear as a solid object from 50 kHz to 100 kHz since currents cannot penetrate the insulating wrap. As we expected by numerical experiments in the previous section, the insulating wrap starts to fade out at frequency 150 kHz and totally disappear at 250kHz. All the experimental results shown in Figures \ref{Fig:exp} (a) are consistent with the numerical simulations in the previous section.
\begin{figure}[ht!]
\centering
\begin{tikzpicture}
  \node at (0,0) {\includegraphics[scale=0.7]{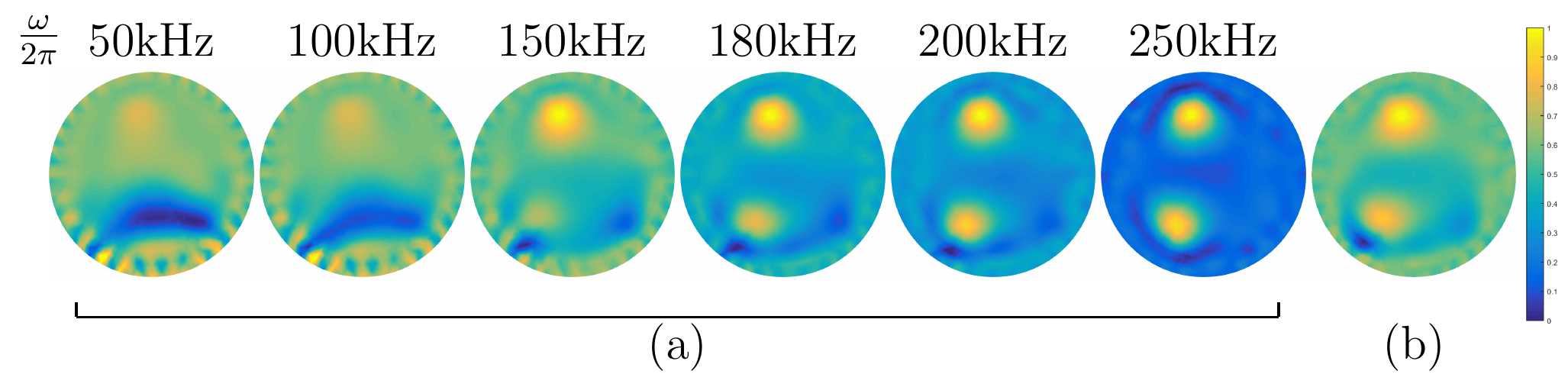}};
\end{tikzpicture}
\caption{(a) Spectroscopic images of normalized admittivity distribution from low frequencies to high frequencies; (b) Fusion of multi-frequency images.}\label{Fig:exp}
\end{figure}
As for the fusion of multi-frequency images, the experiments are conducted at a wide range of frequencies from 50kHz to 250kHz with a step size 10kHz since the minimum frequency step of Swisstom EIT-Pioneer Set is fixed at 10kHz. The integrated image is shown in Figure \ref{Fig:exp} (b) where both two conductive objects and the surrounding insulating wrap are visible.
\section{Conclusion}
In this work, we have provided a rigorous mathematical formula for the jump of potential and normal derivative across the thin linear-shaped insulating inhomogeneities based on layer potential techniques. The potential jump is related with the thickness of insulating objects as well as the current frequencies. Based on these jump conditions, we have developed two asymptotic expansions for current-voltage data perturbations due to thin insulators and small conductors at various frequencies. Using these two asymptotic
expansions, we have mathematically shown that at high frequencies
we can visualize the small conductors, while at low frequencies we
can only get the information of thin insulators. Based on this
mathematical analysis, we conclude that multiple frequencies help
us to handle the spectroscopy behavior of the current-voltage data
with respect to thin insulators and small conductors. When the frequency
increase from very low to very high, we can continuously observe
the images of thin insulators (low frequency), both thin insulators and small conductors (not too low, not too high frequency), and only small insulators (high frequency). The mathematical results are supported by a variety of numerical illustrations and phantom experiments.

\section*{Acknowledgements}
Ammari was supported  by the ERC Advanced Grant Project MULTIMOD--267184. Seo and Zhang were supported by the National Research Foundation of Korea (NRF) grant funded by the Korean government (MEST) (No. 2011-0028868, 2012R1A2A1A03670512)

\bibliographystyle{iopart}

\end{document}